\documentclass[a4paper,11pt]{amsart}

\usepackage{enumerate, booktabs, mathtools}
\usepackage{amsmath, amsfonts, amssymb, amsthm}

\usepackage[T1]{fontenc}
\usepackage[applemac]{inputenc}
\usepackage{mathtools} 
\usepackage{color}
\usepackage[dvipsnames]{xcolor}
\usepackage{relsize}
\usepackage[all]{xy}
\usepackage[height=22.5cm, width=15.5cm]{geometry}
\usepackage[active]{srcltx}
\usepackage{tikz-cd}

\numberwithin{equation}{section}

\theoremstyle{plain}
\newtheorem{thm}{Theorem}[section]
\newtheorem{prop}[thm]{Proposition}
\newtheorem{lem}[thm]{Lemma}
\newtheorem{cor}[thm]{Corollary}
\newtheorem*{Prop3.2}{Proposition 3.2}
\newtheorem*{thm4.2}{Theorem 4.2}

\newtheorem*{thm*}{Theorem}

\theoremstyle{definition}
\newtheorem{defn}[thm]{Definition}

\theoremstyle{remark}
\newtheorem*{rem}{Remark}
\newtheorem*{ex}{Example}

\newcommand{\mbb}[1]{\mathbb{#1}}
\newcommand{\wt}[1]{\widetilde{#1}}
\newcommand{\wh}[1]{\widehat{#1}}

\newcommand{\lie}[1]{{\mathfrak{#1}}}

\renewcommand{\leq}{\leqslant}
\renewcommand{\geq}{\geqslant}

\DeclareMathOperator{\Id}{Id}

\DeclareMathOperator{\End}{End}

\DeclareMathOperator{\Aut}{Aut}

\DeclareMathOperator{\diag}{diag}

\DeclareMathOperator{\kod}{kod}
\DeclareMathOperator{\Alb}{Alb}


\title[Compact complex non-K\"ahler manifolds]{Compact complex non-K\"ahler 
manifolds associated with totally real reciprocal units}

\author{Christian Miebach}
\address{Univ.~Littoral C\^ote d'Opale, UR 2597, LMPA, Laboratoire de 
Math\'ematiques Pures et Appliqu\'ees Joseph Liouville, F-62100 Calais, France}
\email{christian.miebach@univ-littoral.fr}

\author{Karl Oeljeklaus}
\address{Aix-Marseille Univ, CNRS, Centrale Marseille, I2M, UMR 7373,
CMI, 39, rue F.~Joli\-ot-Curie, 13453 Marseille  Cedex 13, France}
\email{karl.oeljeklaus@univ-amu.fr}

\thanks{The authors would like to thank Professor A. Dubickas for helping  
kindly in number theoretical questions in particular for giving 
Lemma~\ref{Lem:reciprocalunit} with a proof. The authors are grateful for 
invitations to the ``Institut de MathŽmatiques de Marseille (I2M)'' and  
``Laboratoire de Math\'ematiques Pures et Appliqu\'ees Joseph Liouville 
(LMPA)'',  where part of this research was carried out.} 

\begin{document}

\begin{abstract}
Using the theory of totally real number fields we construct a new class of 
compact complex non-K\"ahler manifolds in every even complex dimension and 
study their analytic and geometric properties.
\end{abstract}
\bigskip

\dedicatory{Dedicated to Alan T. Huckleberry on the occassion of his 80th birthday}
\maketitle

\section{Introduction}
In this paper we construct a new class of compact complex non-K\"ahler manifolds
in every complex dimension $n=2d, d\geq 1$, and investigate their complex 
analytic and topological properties.

Using totally real number fields we construct first $4d$-dimensional real 
solvable Lie groups $G$ admitting irreducible cocompact discrete subgroups 
$\Gamma$. This method works in general to produce {\it real solv-manifolds}. 
Next we show that these Lie groups admit left invariant complex structures. The 
left quotient $ X:=\Gamma \backslash G$ is then a compact complex manifold. In 
the case $d=1$, one recovers the Inoue surfaces noted $S_{N}^{ \mathsmaller 
{(+)}}$ in the famous paper \cite{In74}, whose extension to higher dimensions 
had remained open since 1974. In 2005, Inoue surfaces of type $S_M$ were 
generalized in~\cite{OT}.

In the following sections we prove that the identity component of the 
holomorphic automorphism group $\Aut^0(X)$ is isomorphic to $(\mbb C^{*})^{d}$ 
and that the whole group $\Aut(X)$ has infinitely many components if $d\geq2$.

The action of $(\mbb C^{*})^{d}$ is free, induces a holomorphic foliation 
$\mathcal F$ which is transversely hyperbolic, and is preserved by the whole 
automorphism group. If $d=2$, the restriction of certain automorphisms of $X$ 
to the tangent bundle $T{\mathcal F}$ has an Anosov property in the sense that 
this bundle splits transversely into a stable and an unstable subbundle.

Furthermore we determine some topological invariants, prove that $X$ is 
non-K\"ahler and show that the algebraic dimension is zero.

Open questions are whether some of the here  constructed manifolds are locally 
conformally K\"ahler, whether they admit proper complex subvarieties, 
 as well as whether they admit Anosov diffeomorphisms relative 
to $\mathcal{F}$ also for $d\geq3$.

\section{The construction}

In this section we explain in detail the construction of a new 
class of compact complex manifolds. In the first two subsections we collect for 
the reader's convenience a number of well-known facts about simply-connected 
nilpotent Lie groups, their rational structures and cocompact discrete 
subgroups, and the free $2$-step nilpotent Lie algebra. In Section~2.3 we 
shall see how particular totally real number fields $K$ allow the construction 
of irreducible rational structures on the $d$-fold product $N$ of the 
three-dimensional real Heisenberg group. Then we extend $N$ by an Abelian group 
in such a way that the corresponding solvable group possesses a cocompact 
discrete subgroup associated with the group of algebraic units in $K$, see 
Section~2.4. Finally, we show that the so obtained solv-manifolds carry a 
complex structure, which completes our construction.

\subsection{Cocompact discrete subgroups of nilpotent Lie 
groups}\label{Section:NilpotentLattices}

In this section we  recall some facts related to cocompact discrete subgroups 
of simply-connected nilpotent Lie groups. For proofs and more details we refer 
the reader to~\cite[Chapter~II]{Rag}.

Let $N$ be a simply-connected nilpotent real Lie group with Lie algebra 
$\lie{n}$. A rational structure on $N$ consists of a rational subalgebra 
$\lie{n}_{\mbb{Q}}$ of $\lie{n}$ such that $\lie{n}_{\mbb{Q}}\otimes_{\mbb{Q}} 
\mbb{R}\cong\lie{n}$. Equivalently, a rational structure on $N$ is given by a 
basis $\mathcal{B}=(\xi_1,\dotsc,\xi_n)$ of $\lie{n}$ such that for all $1\leq 
k<l\leq n$ the coordinates of $[\xi_k,\xi_l]$ with respect to $\mathcal{B}$, 
i.e., the structure constants of $\lie{n}$ with respect to $\mathcal{B}$, are 
rational.

Two rational structures on $N$ are called isomorphic if the corresponding 
rational Lie algebras are isomorphic. A rational structure on $N$ is called 
{\emph{irreducible}} if $\lie{n}_{\mbb{Q}}$ is not isomorphic to the direct sum 
of two non-trivial ideals.

\begin{rem}
There are simply-connected nilpotent Lie groups $N$ that do not admit any 
rational structure. It is also possible that $N$ possesses several 
non-isomorphic rational structures, see~\cite[Remarks~2.14 and 2.15]{Rag}.
\end{rem}

In order to explain how a rational structure on $N$ yields cocompact discrete 
subgroups of $N$, we restate~\cite[Theorem~2.12]{Rag} for the reader's 
convenience. Let $\Lambda\subset\lie{n}$ be any lattice of maximal rank 
contained in $\lie{n}_{\mbb{Q}}$. Then the group $\Gamma$ generated by 
$\exp(\Lambda)$ in $N$ is a cocompact discrete subgroup of $N$. Any two 
discrete subgroups associated with the same rational structure are 
commensurable. Conversely, if $\Gamma\subset N$ is a cocompact discrete 
subgroup, then the $\mbb{Z}$-span of $\exp^{-1}(\Gamma)$ in $\lie{n}$ is a 
lattice of maximal rank in $\lie{n}$ and any basis of $\lie{n}$ contained in 
this lattice defines a rational structure on $N$. If $\wt{\Gamma}$ is 
commensurable with $\Gamma$, then the associated rational structures are 
isomorphic.

\begin{rem}
It follows from the preceding considerations that the rational structure on 
$N$ is irreducible, if and only if the associated cocompact discrete subgroup of $N$ is 
not commensurable to the direct product of two non-trivial normal subgroups.
\end{rem}

\subsection{The free $2$-step nilpotent Lie algebra}

Let $V$ be a $2d$-dimensional real vector space and let $\bigwedge^2V$ denote 
its exterior algebra. On the vector space $V\oplus\bigwedge^2V$ we define a Lie 
bracket by
\begin{equation*}
\bigl[(v,\alpha),(w,\beta)\bigr]:=(0,v\wedge w).
\end{equation*}
The resulting Lie algebra is the free $2$-step nilpotent Lie algebra 
$\lie{f}_{2d}$ of dimension $2d+\left(\begin{smallmatrix}2d\\2\end{smallmatrix} 
\right)=2d^2+d$. We have
\begin{equation*}
W:=\bigwedge^2V=\lie{f}_{2d}'=\mathcal{Z}(\lie{f}_{2d}),
\end{equation*}
where $\mathcal{Z}(\lie{f}_{2d})$ denotes the center of $\lie{f}_{2d}$.

Let $F_{2d}$ be the simply-connected nilpotent Lie group with Lie algebra 
$\lie{f}_{2d}$.

\begin{ex}
The Lie algebra $\lie{f}_{2}$ is isomorphic to the $3$-dimensional Heisenberg 
algebra $\lie{h}_3$. An explicit isomorphism is given by
\begin{equation*}
\lie{h}_3\to\lie{f}_{2},\quad
\begin{pmatrix}
0&x&z\\0&0&y\\0&0&0
\end{pmatrix}
\mapsto(xe_1+ye_2,ze_1\wedge e_2)\in\mbb{R}^2\oplus\bigwedge^2\mbb{R}^2.
\end{equation*}
On the group level, one can realize the $3$-dimensional Heisenberg group as
\begin{equation*}
H_3:=\left\{
\begin{pmatrix}
1&x&z\\0&1&y\\0&0&1
\end{pmatrix};\ x,y,z\in\mbb{R}
\right\}.
\end{equation*}
The map
\begin{equation*}
\begin{pmatrix}
1&x&z\\0&1&y\\0&0&1
\end{pmatrix}
\mapsto\left(x,y,z-\frac{xy}{2}\right)
\end{equation*}
yields an explicit isomorphism with the realization of the Heisenberg group as 
the free nilpotent group $F_2$, with group structure given by 
\[(x,y,z)(\tilde x,\tilde y,\tilde z)=\left(x +\tilde x,y+\tilde 
y,z+\tilde z+\frac{1}{2}(x\tilde y - \tilde x y)\right).\]
We shall  therefore use in the sequel $H_3$ as a model for $F_2$.
\end{ex}

The choice of any basis of $V$ leads to a rational structure on $F_{2d}$ as 
follows. Let $(e_1,\dotsc,e_{2d})$ be a basis of $V$ and put 
$f_{k,l}:=[e_k,e_l]$. Then
\begin{equation*}
(e_1,\dotsc,e_{2d},f_{1,2},f_{1,3},\dotsc,f_{2d-1,2d})
\end{equation*}
is a basis of $\lie{f}_{2d}$ with respect to which the structure constants of 
$\lie{f}_{2d}$ are rational. As explained above, this procedure yields 
cocompact discrete subgroups of $F_{2d}$.

\subsection{Rational structures associated with totally real number fields}

Let $\mbb{Q}\subset K$ be a field extension of degree $2d$, $d\geq1$, and let 
$\mathcal{O}_K \subset K$ be its ring of  algebraic integers. Choose elements 
$\omega_1,\dotsc,\omega_{2d}\in\mathcal{O}_K$ such that
\begin{equation*}
\mathcal{O}_K\cong\mbb{Z}\omega_1\oplus\dotsb\oplus\mbb{Z}\omega_{2d}
\end{equation*}
as $\mbb{Z}$-modules.

We suppose that $K$ is totally real, i.e., that all $2d$ embeddings 
$\sigma_1,\dotsc,\sigma_{2d}\colon K\to\mbb{R}\subset \mbb C$ are real and 
consider the map $\sigma\colon K\to V:=\mbb{R}^{2d}$ given by 
$\sigma(x)=\bigl(\sigma_1(x),\dotsc,\sigma_{2d}(x)\bigr)$. It follows that 
$\Lambda_K:=\sigma(\mathcal{O}_K)\subset V$ is a lattice of maximal rank 
generated by $e_k:=\sigma(\omega_k)$ for $1\leq k\leq 2d$.

Now, as mentioned above, the basis $\mathcal{B}=\bigl(e_{1},\dotsc,e_{2d}\bigr)$ 
of $V$ yields the basis
\begin{equation*}
\wt{\mathcal{B}}=(e_1,\dotsc,e_{2d},f_{1,2},f_{1,3},\dotsc,f_{2d-1,2d})
\end{equation*}
of $ \lie{f}_{2d} $ and induces therefore a rational structure on $F_{2d}$. In 
the following, $(\lie{f}_{2d})_{\mbb{Q}}$ {\emph{always}} denotes the 
corresponding rational Lie algebra.
 
Let $\mathcal{O}^*_K$ be the (multiplicative) group of units in $\mathcal{O}_K$.
We say that a unit $u\in\mathcal{O}^*_K$ is totally positive if $\sigma_j(u)>0$ 
for all $1\leq j\leq 2d$, and we write $\mathcal{O}^{*,+}_K$ for the group of 
totally positive units. Due to Dirichlet's theorem, the group 
$\mathcal{O}^{*,+}_K$ is isomorphic to $\mbb{Z}^{2d-1}$.

The group $\mathcal{O}^{*,+}_K$ acts on $\Lambda_K$ as a group of $\mbb 
Z$-module automorphisms via
\begin{equation*}
u\cdot\sigma(x):=\sigma(ux).
\end{equation*}
Extending $\sigma(x)\mapsto\sigma(ux)$ to an $\mbb{R}$-linear map $\rho(u) 
\colon V\to V$ we obtain a representation
\begin{equation*}
\rho\colon \mathcal{O}^{*,+}_K\to{\rm{SL}}(V).
\end{equation*}
 Note that we have $\det\rho(u)=1$ for all $u\in 
\mathcal{O}^{*,+}_K$ since the eigenvalues of $\rho(u)$ are the conjugates
$\lambda_k=\sigma_{k}(u)$, $k=1,\dotsc,2d$, of $u$ and thus all positive. 

\begin{rem}
Every element of ${\rm{GL}}(V)$ extends uniquely to an automorphism of the Lie 
algebra $\lie{f}_{2d}$. If the matrix of an element in ${\rm{GL}}(V)$ with 
respect to $\mathcal{B}$ has rational coefficients, we obtain an automorphism 
of $(\lie{f}_{2d})_{\mbb{Q}}$.
\end{rem}

Note that by construction the matrix of $\rho(u)$ with respect to $\mathcal{B}$ 
lies in ${\rm{SL}}(2d,\mbb{Z})$ for every $u \in \mathcal O^{*,+}_K$. Hence, 
$\rho(\mathcal{O}^{*,+}_K)$ is a discrete subgroup of ${\rm{SL}}(V)$. 
Furthermore, $\rho(u)$ induces an automorphism of $\lie{f}_{2d}$ that leaves 
$(\lie{f}_{2d})_{\mbb{Q}}$ invariant. We will denote this automorphism of 
$\lie{f}_{2d}$ as well as its restriction to $(\lie{f}_{2d})_{\mbb{Q}}$ by 
$\wh{\rho}(u)$, where 
$\wh{\rho}\colon\mathcal{O}^{*,+}_K\to\Aut\bigl((\lie{f}_{2d})_{\mbb{Q}}
)\bigr)$. Clearly, $\wh{\rho}(u)$ respects the decomposition
\begin{equation*}
(\lie{f}_{2d})_{\mbb{Q}}=V_{\mbb{Q}}\oplus W_{\mbb{Q}} 
\end{equation*}
and the eigenvalues of the restriction of $\wh{\rho}(u)$ to $W_{\mbb{Q}}$ are 
the products $\sigma_{k}(u)\sigma_{l}(u)=\lambda_{k} \lambda_{l}$, $1\leq k<l 
\leq 2d$.

We formulate the following lemma in our setting, but it remains also valid in  a 
more general form. Its proof is a slight adaptation of the proof 
of~\cite[Proposition~2.1.4]{Ash}.

\begin{lem}\label{Lem:simdiagonal}
The set of all endomorphisms $\rho(u)$ of $V_{\mbb{Q}}$, $u\in\mathcal{O}_K^*$  
is simultaneously diagonalizable over the Galois closure $L$ of $K$.
\end{lem}

\begin{proof}
Since the number field $K$ is totally real, there is a primitive unit $u_0\in 
\mathcal{O}^{*,+}$, see e.g. \cite[Theorem~1.4]{ZBA}. Note that $V_{\mbb{Q}}$ 
and $K$ are isomorphic as rational vector spaces and that $\rho(u_0)$ 
corresponds to multiplication by $u_0$. It follows that the characteristic 
polynomial of $\rho(u_0)$ coincides with the minimal polynomial of $u_0$. Hence, 
$\rho(u_0)$ is diagonalizable over $L$ and each of its eigenvalues has 
multiplicity $1$. Since the group $\mathcal{O}^{*,+}$ is Abelian, it stabilizes 
every eigenspace, which proves the claim.
\end{proof}

In order to proceed with the construction, we need the following lemma which 
was communicated to us with proof by Professor A. Dubickas. 
Recall that a unit $u\in\mathcal{O}^{*,+}_K$ is called \emph{reciprocal} if $u$ 
and $u^{-1}$ are conjugate, i.e., have the same minimal polynomial, which then 
is palindromic. Moreover, $u$ is \emph{primitive} if $K=\mbb{Q}(u)$.

\begin{lem}\label{Lem:reciprocalunit} 
For every $d\geq1$ there exist  totally real number fields of degree $2d$ which 
admit primitive reciprocal units. 
\end{lem}

\begin{proof}
Let $\alpha$  be a totally real algebraic integer of degree $d$ and denote its 
minimal polynomial by $P(X)$. Let $L$ be an integer so large that for each of 
the $d$ roots $\alpha_j$, $j=1,\dotsc,d$, of $P(X)$ we have $L+\alpha_j>2$. 
Consider then the polynomial $Q(X)=P(X+\frac{1}{X}-L)X^d$. It is clear that $Q$ 
is a monic palindromic  polynomial of degree $2d$ with $2d$ real roots, since 
$(L+\alpha_j)^2>4$, $j=1,\dotsc,d$. Furthermore, for a generic choice of $L$, 
the polynomial $Q$ is irreducible and its roots are totally real reciprocal 
units of degree $2d$.
\end{proof}

From now on we suppose that there exists a primitive reciprocal unit 
$u_0\in\mathcal{O}^{*,+}_K$. In other words, we suppose that $u_0$ and 
$u_0^{-1}$ are conjugate, as well as $K=\mbb{Q}(u_0)$.

\begin{prop}\label{Prop:directsum}
Let $u_0\in\mathcal{O}^{*,+}_K$ be a primitive reciprocal unit. Then there 
exists a $\wh{\rho}(u_0)$-invariant rational decomposition 
\begin{equation*}
W_{\mbb{Q}}=(W_1)_{\mbb{Q}}\oplus(W_2)_{\mbb{Q}}
\end{equation*}
where $W_1$ is the $d$-dimensional subspace of $\wh{\rho}(u_0)$-fixed points in 
$W$.
\end{prop}

\begin{proof}
Since the characteristic polynomial of $\rho(u_0)$ is palindromic, we can 
arrange the eigenvalues $\lambda_1,\dotsc,\lambda_{2d}$ of $\rho(u_0)$ such 
that $\lambda_{2k}=\lambda_{2k-1}^{-1}$ for all $1\leq k\leq d$. Since 
the restriction of $\wh{\rho}(u_0)$ to $W$ is diagonalizable with 
eigenvalues $\lambda_k\lambda_l$ for $1\leq k<l\leq 2d$, we see that the 
subspace $W_1$ of $\wh{\rho}(u_0)$-fixed points is of dimension $d$. Moreover, 
it follows that the characteristic polynomial of $\wh{\rho}(u_0)|_W$ is 
divisible by $(x-1)^d$ in $\mbb{Z}[x]$. This gives the desired decomposition of 
$W_{\mbb{Q}}$ into two $\wh{\rho}(u_0)$-stable rational subspaces, 
see~\cite[Theorem~XI.4.1]{L}.
\end{proof}

Let us consider the rational Lie algebra
\begin{equation*}
\lie{n}_{\mbb{Q}}:=(\lie{f}_{2d})_{\mbb{Q}}/(W_2)_{\mbb{Q}}.
\end{equation*}
In the following we will view $\lie{n}_{\mbb{Q}}$ as $V_{\mbb{Q}}\oplus 
\bigl(W_{\mbb{Q}}/(W_2)_{\mbb{Q}}\bigr)$. Note that 
$W_{\mbb{Q}}/(W_2)_{\mbb{Q}}$ coincides with the center of $\lie{n}_{\mbb{Q}}$.

\begin{prop}\label{Prop:irreducible}
We have $\lie{n}:=\lie{n}_{\mbb{Q}}\otimes\mbb{R}\cong\lie{h}_3^d$, i.e., the 
construction yields a rational structure on $N:=H_3^d$. Moreover, this rational 
structure on $N$ is irreducible.
\end{prop}

\begin{proof}
Let $(v_1,\dotsc,v_{2d})$ be a basis of $V$ such that $\rho(u_0)v_{2j-1}= 
\lambda_j v_{2j-1}$ and $\rho(u_0)v_{2j}=\lambda_j^{-1}v_{2j}$ for all $1\leq 
j\leq d$. Set $w_{j}:= [v_{2j-1},v_{2j}], \ j=1,\dotsc,d$. Then
\begin{equation*}
\mbb{R}v_{2j-1}\oplus \mbb{R}v_{2j}\oplus\mbb{R}w_{j},\ j=1,\dotsc,d,
\end{equation*}
is a subalgebra of $\lie{f}_{2d}$ isomorphic to $\lie{h}_3$ which intersects 
$W_2$ trivially. Using the decomposition established in 
Proposition~\ref{Prop:directsum}, one sees that these subalgebras commute 
pairwise modulo $W_2$, which proves the first claim.

In order to show that the rational structure on $N$ is irreducible, suppose 
that we have a decomposition $\lie{n}_{\mbb{Q}}=\lie{a}_{\mbb{Q}} 
\oplus\lie{b}_{\mbb{Q}}$ where $\lie{a}_{\mbb{Q}}$ and $\lie{b}_{\mbb{Q}}$ are 
non-trivial ideals of $\lie{n}_{\mbb{Q}}$. Then we have 
$\lie{n}=\lie{h}_3^d=\lie{a}\oplus\lie{b}$ for 
$\lie{a}=\lie{a}_{\mbb{Q}}\otimes_{\mbb{Q}}\mbb{R}$ and $\lie{b}= 
\lie{b}_{\mbb{Q}}\otimes_{\mbb{Q}}\mbb{R}$.

If an ideal of $\lie{n}$ contains an element of the form
\begin{equation*}
\xi=\sum_{j=1}^d\bigr(\kappa_{j}v_{2j-1}+\mu_j v_{2j}+\nu_j w_{j}
\bigr)
\end{equation*}
with $(\kappa_{j_{0}},\mu_{j_0})\not=(0,0)$, then this ideal contains also 
$w_{j_{0}}$. Define
\begin{equation*}
J_{\lie{a}}:=\left\{1\leq j_{0}\leq d;\ \exists\ 
\xi=\sum_{j=1}^d\bigr(\kappa_{j}v_{2j-1}+\mu_j v_{2j}+\nu_j 
w_{j}\bigr)\in\lie{a}\text{ with }(\kappa_{j_{0}},\mu_{j_{0}})\not=(0,0) \right\}
\end{equation*}
and similarly $J_{\lie{b}}$. Since $\lie{n}=\lie{a}\oplus\lie{b}$, the preceding 
observation implies that $J_{\lie{a}}\cup J_{\lie{b}}=\{1,\dotsc,d\}$ and that 
this union is disjoint. Therefore, we can suppose that $J_{\lie{a}}=\{1,\dotsc, 
k\}$ and hence get
\begin{equation*}
\pi(\lie{a})=\bigoplus_{j=1}^k(\mbb{R}v_{2j-1}\oplus\mbb{R}v_{2j})\quad
\text{and}\quad
\pi(\lie{b})=\bigoplus_{j=k+1}^d(\mbb{R}v_{2j-1}\oplus\mbb{R}v_{2j}),
\end{equation*}
where $\pi\colon\lie{h}_3^d\to V$ is the projection along the center of 
$\lie{h}_3^d$.

Since $\pi(\lie{a})$ and $\pi(\lie{b})$ are invariant under $\rho(u_0)$ and 
since $\pi$ is defined over $\mbb{Q}$, we obtain the corresponding rational  
decomposition $V_{\mbb{Q}}=\pi(\lie{a}_{\mbb{Q}})\oplus\pi(\lie{b}_{\mbb{Q}})$ 
into two rational $\rho(u_0)$-invariant subspaces. Now the claim follows 
from the fact that the characteristic polynomial of $\rho(u_0)$ is irreducible 
over $\mbb{Q}$. 
\end{proof}

As explained in Section~\ref{Section:NilpotentLattices} we can now define a 
cocompact discrete subgroup of $N$ as follows.

\begin{defn}
Let $\wh{\Lambda}_K\subset V\oplus   W/W_{2}$ be the full lattice generated by 
$\sigma(\omega_k)\in V$ for $1\leq k\leq 2d$ and the images of 
$f_{kl}=\bigl[\sigma(\omega_k),\sigma(\omega_l)\bigr], k<l$ in 
$W_{\mbb{Q}}/(W_2)_{\mbb{Q}}$ for $1\leq k<l\leq 2d$. We define $\Gamma_N$ to be 
the discrete cocompact subgroup generated by $\exp(\wh{\Lambda}_K)\subset N$.
\end{defn}

\begin{rem}
Due to Proposition~\ref{Prop:irreducible}, the group $\Gamma_N$ is not 
commensurable to the product of two proper normal subgroups.
\end{rem}

 \subsection{Solv-manifolds associated with totally real 
number fields} \label{Subs:CocompactDiscreteSubgroups}

In this subsection, we construct an extension of $N$ by an Abelian group which 
admits a cocompact discrete subgroup containing $\Gamma_N$.

Recall that $\mathcal{O}_K^{*,+}$ can be considered as a discrete Abelian 
subgroup of ${\rm{SL}}(2d,\mbb{R})$ which leaves the lattice $\Lambda_K\subset 
V_{\mbb{Q}}$ invariant. Moreover, its action extends to $V\oplus W$ leaving
$W_{\mbb{Q}}$ invariant.

Now we have
\begin{prop}
The group $\mathcal{O}_K^{*,+}$ respects the decomposition $W_{\mbb{Q}}= 
(W_1)_{\mbb{Q}}\oplus(W_2)_{\mbb{Q}}$. Consequently, every 
$u\in\mathcal{O}_K^{*,+}$ acts on $\lie{n}_{\mbb{Q}}$ by an automorphism 
$\rho_{\lie{n}}(u)$.
\end{prop}

\begin{proof}
Due to Lemma~\ref{Lem:simdiagonal}, the transformations $\rho(u)\in \End(V)$ 
with $u\in\mathcal{O}^{*,+}_K$ are simultaneously diagonalizable. Thus the 
same holds for the transformations $\wh{\rho}(u)\in{\rm{GL}}(V\oplus W)$. 
The claim follows from the fact that $W_1$ and $W_2$ are direct sums of 
eigenspaces for $\wh{\rho}(u_0)$. In fact, $W_1$ is the eigenspace corresponding 
to the eigenvalue $1$ and $W_2$ is a direct sum of eigenspaces with eigenvalues 
not equal to $1$. 
\end{proof}

Since $\mathcal{O}_K^{*,+}$ acts on $\lie{n}_{\mbb{Q}}$ by automorphisms, it 
respects the decomposition $\lie{n}_{\mbb{Q}}=V_{\mbb{Q}}\oplus 
\mathcal{Z}(\lie{n}_{\mbb{Q}})$ where $\mathcal{Z}(\lie{n}_{\mbb{Q}})\cong 
(W_1)_{\mbb{Q}}$ as rational vector spaces. Consider the homomorphism
\begin{equation*}
\psi\colon\mathcal{O}_K^{*,+}\to{\rm{SL}}\bigl(\mathcal{Z}(
\lie{n}_{\mbb{Q}})\bigr),\quad 
\psi(u):=\rho_{\lie{n}}(u)|_{\mathcal{Z}(\lie{n}_{\mbb{Q}})}.
\end{equation*}
Note that by construction we have $u_0\in\ker(\psi)$.

\begin{prop}
The subgroup $\ker(\psi)\subset\mathcal{O}^{*,+}_K$ has rank $d$, hence 
is isomorphic to $\mbb{Z}^d$, and consists of reciprocal units in 
$\mathcal{O}^{*,+}_K$.
\end{prop}

\begin{proof}
Let $T_n$ denote the group of diagonal matrices in ${\rm{SL}}(n,\mbb{R})$
having strictly positive entries.

Due to Lemma~\ref{Lem:simdiagonal}, there exists an element
$g_0\in{\rm{SL}}(2d,\mbb{R})$ such that $g_0\mathcal{O}^{*,+}_Kg_0^{-1}$
is a discrete subgroup of $T_{2d}$. Every element $g_0^{-1}\diag(t_1,\dotsc,t_{2d})g_0$ of $\wt{A}:=g_0^{-1}T_{2d}g_0$ induces a
linear transformation on $W_1$ that is again diagonalizable with 
eigenvalues $(t_1t_2,t_3t_4,\dotsc,t_{2d-1}t_{2d})$. This yields a 
homomorphism  $\wh{\psi}\colon\wt{A}\to{\rm{SL}}(d,\mbb{R})$ which 
extends $\psi$ and whose image is conjugate to a subgroup of $T_d$.

According to Dirichlet's theorem, we can view $\psi$ as a homomorphism
from $\mbb{Z}^{2d-1}$ to ${\rm{SL}}(d,\mbb{Z})$. In particular, the image of
$\psi$ is conjugate to a \emph{discrete} subgroup of $T_d$ and therefore
has rank at most $d-1$. This implies that the rank of $\ker(\psi)$ is at
least $d$.

The Lie algebra $\wt{\lie{a}}$ of $\wt{A}$ is conjugate to the set of
trace zero diagonal matrices, hence $\wt{\lie{a}}\cong\mbb{R}^{2d-1}$. The 
derivative of $\wh{\psi}\colon\wt{A}\to{\rm{SL}}(d,\mbb{R})$ can be
identified with the map $\mbb{R}^{2d-1}\to\mbb{R}^{d-1}$ given by
\begin{equation*}
(x_1,\dotsc, x_{2d})\mapsto(x_1+x_2, 
x_3+x_4, \dotsc,x_{2d-1}+x_{2d}),
\end{equation*}
where we suppose that $x_1+\dotsb+x_{2d}=0$. Since this map is 
surjective, its kernel is isomorphic to $\mbb{R}^d$, which implies that
the \emph{discrete} subgroup $\ker(\psi)\subset A:=\ker(\wh{\psi})$
is of rank at most $d$.
\end{proof}
 
Let us summarize our construction. We have seen that we can view 
$\mathcal{O}^{*,+}_K\cong\mbb{Z}^{2d-1}$ as a discrete subgroup of 
${\rm{SL}}(2d,\mbb{R})$ that normalizes $\Gamma_N$. The identity component
of its real Zariski closure is $\wt{A}\cong(\mbb{R}^{>0})^{2d-1}$ in 
${\rm{SL}}(2,\mbb{R})$, the elements of which are simultaneously 
diagonalizable. Moreover, the identity component of the real Zariski 
closure of $\Gamma_A:=\ker(\psi)\cong\mbb{Z}^d$ is
\begin{equation*}
A\cong\bigl\{(a_1,b_1,\dotsc,a_d,b_d)\in(\mbb{R}^{>0})^{2d};\ 
a_1b_1=\dotsb=a_d b_d=1\bigr\}\cong(\mbb{R}^{>0})^d.
\end{equation*}
Consequently, $\Gamma_A$ acts on $\Gamma_N$ and we obtain the solvable
discrete subgroup $\Gamma:=\Gamma_A\ltimes\Gamma_N$ which is cocompact in
$A\ltimes N\cong(\mbb{R}^{>0})^d\ltimes N$ and Zariski dense in 
$(\mbb{R}^*)^d\ltimes N$. 

Since for $a,b \in \mbb R^{>0}$ and $x,y,z \in\mbb R$ we have
\begin{equation*}
\begin{pmatrix}ab & 0 & 0 \\0 & b & 0 \\0 & 0 & 1\end{pmatrix}
\begin{pmatrix}1 & x & z \\0 & 1 & y \\0 & 0 & 1\end{pmatrix}
\begin{pmatrix}a^{-1}b^{-1} &0 & 0 \\0 & b^{-1} & 0 \\0 & 0 & 1 \end{pmatrix}=
\begin{pmatrix}1 & ax & abz \\0 & 1 & by \\0 & 0 & 1\end{pmatrix},
\end{equation*}
one can realise the Lie group ${\wt G}:={\wt A}\ltimes N\cong 
(\mbb{R}^{>0})^{2d-1} \ltimes N$ as a matrix group isomorphic to
\begin{equation*}
\left\{ \left.\begin{pmatrix}
M_{1 } & 0 & \cdots & 0 \\
0& M_{2 } & \cdots & 0\\
\vdots  & \vdots  & \ddots & \vdots  \\
0& 0& \cdots & M_{d} 
\end{pmatrix}   \right| \   M_{i}=
 \begin{pmatrix}
a_{i} b_{i}&x_{i}&z_{i}\\0&b_{i}&y_{i}\\0&0&1
\end{pmatrix},\ a_{i},b_{i}\in\mbb{R}^{>0},x_{i},y_{i},z_{i}\in\mbb{R}, \
i=1,\dotsc,d  \right\}.
\end{equation*} 
Under this isomorphism  the group  $\mathcal{O}^{*,+}_K $ corresponds to
\begin{equation*}
\left\{\left. \begin{pmatrix}
D_{1 } & 0 & \cdots & 0 \\
0& D_{2 } & \cdots & 0\\
\vdots  & \vdots  & \ddots & \vdots  \\
0& 0& \cdots & D_{d} 
\end{pmatrix}   \right| \   D_{i}=
 \begin{pmatrix}
\sigma_{2i-1}(u) \sigma_{2i}(u)&0&0\\0&\sigma_{2i}(u)&0\\0&0&1
\end{pmatrix},  \ i=1,\dotsc,d, \ u \in  \mathcal{O}^{*,+}_K 
 \right\}.
\end{equation*}
Furthermore $G:=A\ltimes N\cong(\mbb{R}^{>0})^d\ltimes N$ is the subgroup of
${\wt G} :={\wt A}\ltimes N   $ given by the equations  $a_{i} b_{i}=1,   \ i 
= 1,\dotsc,d$   
and the subgroup $\Gamma_{A} \subset \mathcal{O}^{*,+}_K$
corresponds to
\begin{equation}\label{Eqn:MultiplicativeAction}
\left\{ \left.\begin{pmatrix}
D_{1 } & 0 & \cdots & 0 \\
0& D_{2 } & \cdots & 0\\
\vdots  & \vdots  & \ddots & \vdots  \\
0& 0& \cdots & D_{d} 
\end{pmatrix}   \right| \   D_{i}=
 \begin{pmatrix}
1&0&0\\0&\sigma_{2i}(u)&0\\0&0&1
\end{pmatrix},  \ i=1,\dotsc,d, \ u \in  A
 \right\}.
\end{equation}

\subsection{Left-invariant complex structure on $G$}

The matrix group 
\begin{equation*}
S:=\mbb{R}^{>0}\ltimes H_3=\left\{\left.
\begin{pmatrix}1 & x & v \\0 & b &  y \\0 & 0 & 1\end{pmatrix}\right|\  
b\in \mbb R^{>0}, \ x,y,v \in \mbb R \right\}
\end{equation*}
acts linearly on $\mbb{C}^3$. The affine hyperplane $\mbb{C}^2\times\{1\}$ of 
$\mbb{C}^3$ is invariant under $S$. A direct calculation shows that the orbit 
through the point $z=(0,i,1)$ is open, has trivial isotropy and coincides with 
$\mbb{C}\times\mbb{H}^+\times\{1\}$ where $\mbb{H}^+\subset \mbb{C}$ is the 
upper half plane.

This proves the following result.

\begin{prop}
The solvable real Lie group $S$ admits a left-invariant complex structure with 
respect to which it is biholomorphic to $\mbb{C}\times\mbb{H}^+$.
\end{prop}

Consider now the natural action of ${\wt G}$ on $\mbb C^{3d}$. The orbit of the 
subgroup $G$ through the point $(0,i,1,0,i,1,\dotsc,0,i,1)\in\mbb{C}^{3d}$ has 
trivial isotropy group, is biholomorphic to $(\mbb C \times \mbb{H}^+)^{d}$ 
and hence gives a left-invariant complex structure on the real Lie group $G$.

Therefore, the left quotient $X:=\Gamma\backslash G$ is a compact complex 
manifold.

\subsection{A density property of $\Gamma_N$}

We continue to consider $N\cong H_3^d$ and $G\cong S^d$ as matrix groups 
consisting of block diagonal matrices as written down in the closing of 
Section~\ref{Subs:CocompactDiscreteSubgroups}. 

Let 
\begin{equation*}
H:=\left\{\left. \begin{pmatrix}
M_{1 } & 0 & \cdots & 0 \\
0& M_{2 } & \cdots & 0\\
\vdots  & \vdots  & \ddots & \vdots  \\
0& 0& \cdots & M_{d} 
\end{pmatrix}   \right| \   M_{i}=
 \begin{pmatrix}
1&x_{i}&z_{i}\\0&1&0\\0&0&1
\end{pmatrix} ,x_{i},v_{i}\in\mbb{R}, \
i=1,\dotsc,d  \right\}.
\end{equation*} 
 
\begin{prop}\label{Prop:dense}
The subgroup $\Gamma_{N} H$ is topologically dense in the Lie group $N$.
\end{prop}

\begin{proof}
Let $K:=(\overline{ \Gamma_{N} H}^{\, \rm top})^{0}$ be the identity component 
of the topological closure of $\Gamma_{N} H$ in $N$. Since $\Gamma_{N} H$ is 
invariant under conjugation by elements of $\Gamma_{A}$, we see that the same is 
true for the subgroup $K$ which has also the property that $(K\cap 
\Gamma_{N})\backslash K$ is compact. Therefore the Lie algebra $\mathfrak k$ of 
$K$ is compatible with the rational structure of $\mathfrak n$ and the 
projection $V'$  of $\mathfrak k$ along the center $\mathfrak z(\mathfrak n)$  
onto the vector space $V$ has the same property. Furthermore, $V' \subset V$ is 
a $u_{0}$-invariant non-trivial subspace compatible with the rational structure 
and this implies that $V' = V$, see also the proof of 
Proposition~\ref{Prop:irreducible}. The proposition is proven.
\end{proof}

\section{Properties of the quotient manifolds}

Let $G=A\ltimes N=(\mbb{R}^{>0}\ltimes H_3)^d=S^{d}$ be the solvable Lie group 
equipped with the left-invariant complex structure such that $G\cong 
(\mbb{C}\times\mbb{H})^d$ and let $\Gamma=\Gamma_A\ltimes\Gamma_N$ be the 
cocompact discrete subgroup constructed above. In this section we establish a 
number of topological and complex geometric properties of the compact complex 
manifold $X=\Gamma\backslash G$.

\subsection{The CR-fibration with Levi-flat fibers, the transversally 
hyperbolic foliation $\mathcal F$ and the Kodaira 
dimension}\label{Subs:CRfibration}

Let $Z$ denote the center of $G$, which is also the center of $N$. We first 
remark that $X$ considered as a real solv-manifold admits the following 
commutative diagram of equivariant fibrations, see~\cite[Proposition 2.17, 
Theorem~3.3, and Corollary~3.5]{Rag}:

\begin{equation*}
\begin{tikzcd}
 & Z\cdot \Gamma \backslash G \arrow{dr}{(S_{1})^{2d}} \\
p\colon X=\Gamma \backslash G \arrow{ur}{(S_{1})^{d}} \arrow{rr}{\Gamma_{N} 
\backslash N} &&  N \cdot \Gamma \backslash G \cong (S_{1})^{d}
\end{tikzcd}
\end{equation*}

The group $\mbb{C}^d$ acts on $G\cong(\mbb{C}\times\mbb{H})^d$ by translation in 
the $\mbb{C}$-factors. One shows directly that this action commutes with the 
left multiplication by $G$ and hence induces a holomorphic action of $\mbb{C}^d$ 
on $X$. The ineffectivity of this action is $\Gamma\cap Z$ and therefore we 
obtain an inclusion $(\mbb{C}^*)^d\hookrightarrow \Aut(X)$. The orbits of this  
$(\mbb{C}^*)^d$ are exactly the images of the $\mbb{C}^d$-factors in the 
universal covering of $X$. As a consequence, we see that the action of 
$(\mbb{C}^*)^d$ on $X$ induces a transversally hyperbolic holomorphic foliation 
$\mathcal{F}$ of $X$.

Since the lift of $p$ to the universal covering $G$ of $X$ coincides 
with the quotient map $G\to G/N\cong{\mbb{R}}^d$, we see that $p$ is 
$(\mbb{C}^*)^d$-invariant. Moreover, the construction of the left-invariant 
complex structure on $G$ shows that the $N$-orbits are generic CR-submanifolds 
of real dimension $3d$ and CR-dimension $d$ in $G$. Since the complex tangent 
space to the $N$-orbits contains the $N'$-orbit, they are Levi-flat. It follows 
that $p$ is a CR-map having Levi-flat fibers.

We determine the topological closure of the orbits of this 
$(\mbb{C}^*)^d$-action in the following

\begin{prop}\label{Prop:closure}
Let $x=\Gamma g\in \Gamma\backslash G=X$. The topological closure of 
$(\mbb{C}^*)^d\cdot x$ in $X$ coincides with the fiber of the projection $p$ 
passing through the point $x$ and is therefore isomorphic to the CR-nilmanifold 
$\Gamma_{N}\backslash N$.  In particular, $X$ does not contain 
any proper $(\mbb{C}^*)^d$-invariant analytic subset.
\end{prop}

\begin{proof}
For the proof, it suffices to remark that the $(\mbb{C}^*)^d$-orbits are 
exactly the right orbits in $\Gamma\backslash G=X$ of the normal subgroup $H$ 
and to apply Proposition~\ref{Prop:dense}.
\end{proof}

\begin{cor}
Every holomorphic function on $\Gamma_N\backslash G$ is constant.
\end{cor}

\begin{proof}
This follows from Proposition~\ref{Prop:closure}, since $\Gamma_{N} \backslash 
N$ is a generic CR-submanifold of $\Gamma_{N} \backslash G$.
\end{proof}

This corollary implies the following

\begin{cor}
The Kodaira dimension of $X$ is $-\infty$.
\end{cor}

\begin{proof}
Since $\Gamma$ acts by affine-linear transformations on 
$(\mbb{C}\times\mbb{H})^d$, the tangent bundle of $X$ and all its induced vector 
bundles are flat. In particular, the canonical bundle of $X$ and all its powers 
are flat, i.e., given by representations of $\Gamma $ in $\mbb C^{*}$. This 
implies that the canonical bundle of a finite covering of $\Gamma_N\backslash G$ 
is holomorphically trivial, since for the commutator group one has $\Gamma' 
\subset \Gamma_{N}$. Since every holomorphic function on $\Gamma_N\backslash G$ 
is constant, we see that $H^0(X,K_X^n)=0$ for all $n\geq0$. Hence, $\kod 
X=-\infty$.
\end{proof}

\subsection{The identity component of $\Aut(X)$ and the non-K\"ahler property}

In order to determine explicitly the holomorphic vector fields on $X$, let us 
give the action $G=A\ltimes N$ on each factor of $(\mbb{C}\times\mbb{H}^+)^d$ 
explicitly. For $g=\left(\begin{smallmatrix}1&a&c 
\\0&t&b\\0&0&1\end{smallmatrix}\right)\in\mbb{R}^{>0}\ltimes H_3$ and 
$(z,w)\in\mbb{C}\times\mbb{H}^+$ we have
\begin{equation*}
\begin{pmatrix}
1&a&c\\0&t&b\\0&0&1
\end{pmatrix}\cdot(z,w)=(z+aw+c,tw+b).
\end{equation*}
Let $\pi\colon G\to\Aut\bigl((\mbb{H}^+)^d\bigr)$ be the natural projection. 
It follows from Proposition~\ref{Prop:dense} that $\pi(\Gamma_N)$ is a 
countable, topologically dense subgroup of the unipotent radical of the Borel 
subgroup of affine transformations in $\Aut\bigl((\mbb{H}^+)^d\bigr)$.  
This observation allows us to carry over the proof 
of~\cite[Proposition~3(ii)]{In74} in order to obtain the following.

\begin{prop}
We have $H^0(X,\Theta)=\mbb{C}^d \cong \langle \frac{\partial}{\partial 
z_{1}},\dotsc,\frac{\partial}{\partial z_{d}}\rangle_{\mbb C} $ and therefore 
$\Aut^0(X)\cong(\mbb{C}^*)^d$.
\end{prop}

\begin{cor}
The manifold $X$ is not K\"ahler.
\end{cor}

\begin{proof}
If $X$ was K\"ahler, then due to~\cite{Fuj} the group $\Aut^0(X)$ would act 
meromorphically on $X$ and consequently its orbits would be locally closed, 
which is not the case.
\end{proof}

\subsection{Infinitely many connected components for $d \geq 2$}

In this subsection we show that the whole group $\Aut(X)$ has infinitely 
many components for $d\geq2$. Note that the automorphism groups of Inoue 
surfaces $S_{N}^{ \mathsmaller{(+)}}$, (this is the case $d=1 $) have only 
finitely many components.

First we note that the group $\wt{A}\cong(\mbb{R}^{>0})^{2d-1}$ acts as a group 
of holomorphic transformations on $(\mbb{C}\times\mbb{H}^+)^d$ by
\begin{equation*}
(\lambda_1,\mu_1,\dotsc,\lambda_d,\mu_d)\cdot (z_1,w_1,\dotsc,z_d,w_d) := 
(\lambda_1\mu_1 z_1,\mu_1 w_1,\dotsc,\lambda_d\mu_d z_d,\mu_d w_d),
\end{equation*}
where we suppose $\lambda_1\mu_1\dotsb\lambda_d\mu_d=1$. 

This action extends the $A$-action on $(\mbb{C}\times\mbb{H}^+)^d$ where  $A$  is embedded in $\wt{A}$ by
\begin{equation*}
(\mu_1,\dotsc,\mu_d)\mapsto(\mu_1^{-1},\mu_1,\dotsc,\mu_d^{-1},\mu_d).
\end{equation*}

In the next step we show that the $\wt{A}$-action normalizes the simply 
transitive $G$-action on $(\mbb{C}\times\mbb{H}^+)^d$. Writing $G=A\ltimes N$ as 
the $d$-fold product of the matrix group
\begin{equation*}
S=\left\{
\begin{pmatrix}
1 & a & c\\ 0 & \alpha & b \\ 0 & 0 & 1
\end{pmatrix};\ \alpha\in\mbb{R}^{>0}, a,b,c\in\mbb{R}\right\},
\end{equation*}
and defining $\varphi_{\lambda,\mu}(z,w):=(\lambda\mu z,\mu w)$ for 
$\lambda,\mu>0$ and $(z,w)\in\mbb{C}\times\mbb{H}^+$, we obtain
\begin{equation*}
\varphi_{\lambda,\mu}\bigl(g\cdot\varphi_{\lambda,\mu}^{-1}(z,w)\bigr)=
(z+\lambda a w + \lambda\mu c,\alpha w+\mu b).
\end{equation*}
Consequently, the induced action of $\wt{A}$ on $G$ coincides with the 
conjugation of $\wt{A}$ on the normal subgroup $A\ltimes N$ in $\wt{A} \ltimes 
N$.

It therefore follows that the action of the subgroup $\mathcal{O}^{*,+}_K$ of 
$\wt{A}$ on $(\mbb{C}\times\mbb{H}^+)^d$ normalizes $\Gamma=\Gamma_A 
\ltimes\Gamma_N$. This implies that the action of $\mathcal{O}^{*,+}_K$ descends 
holomorphically to the compact quotient $X=\Gamma\backslash G$.

\subsection{An Anosov property of the foliation $\mathcal F$ in the case $d=2$}

As we have seen in the previous subsection, the group $\mathcal{O}^{*,+}_K/ 
\Gamma_A$ embeds into $\Aut(X)$. It is clear that this discrete group of 
automorphisms stabilizes the foliation $\mathcal{F}$ of $X$. In this subsection 
we shall see  that for $d=2$ non-trivial elements $\varphi$ of $\mathcal{O}^{*,+}_K/ 
\Gamma_A$ 
have an Anosov property relative to $\mathcal{F}$, i.e. that the 
bundle map  $\varphi_*$ is Anosov when restricted to the involutive subbundle 
$T\mathcal{F}\subset TX$.  

Suppose first that $d$ is arbitrary and consider the bundle map $\varphi_*\colon 
TX\to TX$ given by the push-forward of tangent vectors. We shall trivialize 
first $TG$ via left-invariant vector fields which trivialize then $TX$ as well. 
Concretely, let $g\in G$ and consider $\varphi_*\colon T_gG\to T_{\varphi(g)}G$. 
Since $T_gG=(\ell_g)_*\lie{g}$, we are led to consider
\begin{equation*}
\bigl(\ell^{-1}_{\varphi(g)}\circ\varphi\circ\ell_g\bigr)_*\colon 
\lie{g}\to\lie{g}.
\end{equation*}
The map $G\to{\rm{GL}}(\lie{g})$ given by $g\mapsto 
\bigl(\ell^{-1}_{\varphi(g)}\circ\varphi\circ\ell_g\bigr)_*$ encodes the action 
of $\varphi_*$ on $TG$. Moreover, since $\varphi$ normalizes the action of 
$\Gamma$ by left multiplication on $G$, it follows that we obtain a well-defined 
map
\begin{equation*}
\rho_\varphi\colon X=\Gamma\backslash G\to{\rm{GL}}(\lie{g})
\end{equation*}
that encodes the action of $\varphi_*$ on $TX$. In particular, for 
$\varphi=\ell_{\gamma}$ with $\gamma\in\Gamma$ we have 
$\rho_\varphi(x)=\Id_\lie{g}$ for all $x\in X$.

Since  the above defined matrix group  $S$ is an open subset of an affine subspace of $\mbb{R}^{3\times3}$, we 
have global coordinates on $G=S^d$ with respect to which we can explicitly 
calculate the map $S^d\to{\rm{GL}}(\lie{g})$. For $\varphi\colon S^d\to S^d$ 
given by
\begin{equation*}
\varphi
\begin{pmatrix}
1&x_i&z_i\\0&a_i&y_i\\0&0&1
\end{pmatrix}=
\begin{pmatrix}
1&\lambda_i\mu_i x_i&\lambda_i\mu_i z_i\\0&\mu_i a_i&\mu y_i\\0&0&1 
\end{pmatrix},   \ \ i=1,\dotsc,d,
\end{equation*}
and $\xi_i=\left(\begin{smallmatrix}0&p_i&r_i\\0&t_i&q_i\\0&0&0\end{smallmatrix}
\right)\in\lie{s}$ we obtain
\begin{equation*}
\bigl(\ell^{-1}_{\varphi(g)}\circ\varphi\circ\ell_g\bigr)_*\xi_i=
\begin{pmatrix}
0&\lambda_i\mu_i p_i&\lambda_i\mu_i r_i\\0&t_i&q_i\\0&0&0
\end{pmatrix}.
\end{equation*}
This shows that for all $\varphi \in  \mathcal{O}^{*,+}_K/ \Gamma_A $ the action 
of $\varphi$ on the bundle $T\mathcal F$ is given by multiplication with the 
$\lambda_i\mu_i$ in the coordiante $\frac{\partial}{\partial z_{i}}$ for 
$i=1,\dotsc,d$.

If $d=2$, let $\varphi\in\mathcal{O}^{*,+}_K/\Gamma_A$ be a non-trivial element 
and let $\lambda_{1},\mu_{1},\lambda_{2},\mu_{2} >0$ be the factors 
corresponding to $\varphi$. We have $\prod_{i=1}^{2}  \lambda_{i}\mu_{i} =1$. 
If one of the products $\lambda_{i}\mu_{i}$ was equal to $1$, the other one 
would also be equal to $1$. Then $\varphi$ would be an element of $\Gamma_{A}$ 
and, considered as an element of $\Aut(X)$, would be the identity, a 
contradiction. Therefore the bundle $T\mathcal F$ enjoys the mentioned Anosov 
property with respect to $\varphi$.

For $d\geq 3$ it seems to be an interesting number theoretic question if there 
always exists an automorphism having this Anosov property with respect to the 
foliation $\mathcal F$.

\subsection{Topological structure of $X$}

Since $G$ is simply-connected solvable and since the adjoint operators of $A$ 
are diagonalizable over $\mbb{R}$, the real cohomology of $X$ may be computed 
via the Lie algebra cohomology of $\lie{g}$.
 
Since $G=A\ltimes N$ is the identity component of the real-algebraic Lie group 
$(\mbb{R}^*)^d\ltimes N$ and since $\Gamma$ is Zariski-dense in 
$(\mbb{R}^*)^d\ltimes N$, we may apply~\cite[Corollary~7.29]{Rag} in order to 
determine the deRham cohomology of $X$.
 
\begin{prop}
We have $H^k(X,\mbb{R})\cong H^k(\lie{g})$ for all $k\geq0$. Moreover, we have
\begin{equation*}
H^k(\lie{g})=\bigoplus_{k_1+\dotsb+k_d=k}\bigl(H^{k_1}(\mbb{R}\oplus\lie{h}_3) 
\otimes \dotsb \otimes H^{k_d}(\mbb{R}\otimes\lie{h}_3)\bigr),
\end{equation*}
where $H^0(\mbb{R}\otimes\lie{h}_3)\cong H^4(\mbb{R}\otimes\lie{h}_3) 
\cong\mbb{R}$, $H^1(\mbb{R}\otimes\lie{h}_3)\cong H^3(\mbb{R}\otimes\lie{h}_3) 
\cong\mbb{R}$, and $H^2(\mbb{R}\otimes\lie{h}_3)=\{0\}$.
\end{prop}

\begin{rem}
The above result shows that the topological Euler characteristic of $X$ is zero. 
This can also be directly deduced from the fact that $X$ is diffeomorphic to a 
tower of torus bundles over $(S^1)^d$. \ More precisely, the 
projection $G=A\ltimes N\to A$ induces a real fiber bundle $X\to \Gamma_A 
\backslash A\cong (S^1)^d$ with typical fiber $\Gamma_N\backslash N$ which in 
turn has the structure of a smooth fiber bundle over $(S^1)^d$ with fiber 
$(S^1)^{2d}$, cf.~the diagram in Section~\ref{Subs:CRfibration}. 
\end{rem}

\subsection{Closed holomorphic $1$-forms on $X$}
 
In this subsection we give a second proof of the fact that $X$ is not K\"ahler.

\begin{prop}\label{Prop:closed1forms}
There is no non-zero closed holomorphic $1$-form on $X$. In particular, $X$ 
is not K\"ahler.
\end{prop}

\begin{proof}
Let $\omega$ be a closed holomorphic $1$-form on $X$ and let $\xi$ be a 
holomorphic vector field on $X$ induced by the $(\mbb{C}^*)^d$-action. Then we 
have
\begin{equation*}
\mathcal{L}_\xi(\omega)=\iota_\xi d\omega+d\iota_\xi\omega=0.
\end{equation*}
Hence, every closed holomorphic $1$-form on $X$ must be 
$(\mbb{C}^*)^d$-invariant. Pulling it back we get a $\Gamma$-invariant closed 
holomorphic $1$-form $\omega$ on $(\mbb{C}\times\mbb{H})^d$ which must be of 
the form
\begin{equation*}
\omega=\sum_{j=1}^d\lambda_jdz_j+\sum_{j=1}^d f_j(w_1,\dotsc,w_d)dw_j.
\end{equation*}
Since an element of $N$ acts on $dz_j$ by $dz_j\mapsto dz_j+adw_j$, we conclude 
that in fact
\begin{equation*}
\omega=\sum_{j=1}^d f_j(w_1,\dotsc,w_d)dw_j.
\end{equation*}
Now the claim follows from the fact that $\pi(\Gamma)$ 
contains a dense subgroup of the unipotent radical of the 
Borel subgroup of affine transformations in $\Aut\bigl((\mbb{H}^+)^d\bigr)$, 
see Proposition~\ref{Prop:dense}. 
\end{proof}

\subsection{The algebraic dimension of $X$}

We conclude by determining the algebraic dimension of $X$.

\begin{thm}
Every meromorphic function on $X$ is constant, i.e., $X$ has algebraic dimension 
zero.
\end{thm}

\begin{proof}
Due to~\cite[Example~2]{Car}, there exists a projective complex space $Y$ and 
a holomorphic map $\pi\colon X\to Y$ such that every holomorphic map from $X$ 
to any projective complex space facto\-rizes through $\pi$. Consequently, 
$(\mbb{C}^*)^d$ acts holomorphically on $Y$ such that $\pi$ is equivariant.

We claim that the induced action of $(\mbb{C}^*)^d$ on $\Alb(Y)$ is trivial. If 
this was not the case, the composed map $X\to Y\to\Alb(Y)$ would not be 
constant and hence we would obtain a non-zero closed holomorphic $1$-form on 
$X$, contradicting Proposition~\ref{Prop:closed1forms}.

Since $(\mbb{C}^*)^d$ acts trivially on $\Alb(Y)$, it has a fixed point in $Y$ 
due to~\cite{Som}; the $\pi$-fiber over this fixed point is a 
$(\mbb{C}^*)^d$-invariant analytic subset of $X$, hence $X$ itself due 
to~Proposition~\ref{Prop:closure}. It follows that $Y$ is a point, i.e., every 
holomorphic map from $X$ to a projective complex space is constant.

Now let us consider the algebraic reduction $a\colon X\dashrightarrow Z$ which 
is a priori only a meromorphic map. Since there are no proper 
$(\mbb{C}^*)^d$-invariant analytic subsets of $X$, we obtain a holomorphic map 
$X\to\mbb{P}_N$ by adding sufficiently many meromorphic functions. As we have 
seen above, this map must be constant, which proves the claim.
\end{proof}

\subsection{Possible extensions of the construction}

In the same spirit as in Inoue's original paper, we can modify the 
multiplicative action of $\Gamma_A$ as described in 
equation~\eqref{Eqn:MultiplicativeAction} as follows.

Firstly, every generator of $\Gamma_A$ may be combined with a central element 
that acts by translation on $\mbb{C}^d$, compare the definition of $g_0$ in 
equation~(18) on page~276 in~\cite{In74}.

Furthermore, we can construct analogs of the surfaces $S_N^{(-)}$ by adding any
number of minus signs in the coordinates $(z_1,\dotsc,z_d)\in\mbb{C}^d$, 
compare the definition of $g_0$ in equation~(21) on page~279 in~\cite{In74}. Of 
course, such a choice will diminish the dimension of the automorphism group of 
the resulting quotient manifold.

\end{document}